\documentclass{article}
\usepackage{amsmath,amsfonts,amssymb,amsthm}
\usepackage{natbib}
\usepackage{booktabs}
\usepackage{xcolor}
\usepackage{paralist}
\usepackage{xspace}
\usepackage{graphicx}
\usepackage[hyphens]{url}
\usepackage{hyperref}

\usepackage[paperwidth=8.5in, textwidth=6in]{geometry}
\usepackage{setspace}

\newtheorem{theorem}{Theorem}

\usepackage[normalem]{ulem} % adds the \sout command

\renewcommand{\le}{\leqslant}
\renewcommand{\ge}{\geqslant}

\DeclareSymbolFont{GreekLetters}{OML}{cmr}{m}{it} %Provide missing letters
\DeclareSymbolFont{UpSfGreekLetters}{U}{cmss}{m}{n} %Provide missing letters
\DeclareMathSymbol{\varrho}{\mathalpha}{GreekLetters}{"25}
\DeclareMathSymbol{\UpSfLambda}{\mathalpha}{UpSfGreekLetters}{"03}
\DeclareMathSymbol{\UpSfSigma}{\mathalpha}{UpSfGreekLetters}{"06}

\newlength{\overwdth}

\makeatletter
\newcommand*\bigcdot{\mathpalette\bigcdot@{.7}}
\newcommand*\bigcdot@[2]{\mathbin{\vcenter{\hbox{\scalebox{#2}{$\m@th#1\bullet$}}}}}
\makeatother

\DeclareMathOperator{\var}{var}
\DeclareMathOperator{\opt}{opt}

\newcommand{\reals}{\mathbb{R}}

\newcommand{\cf}{\mathcal{F}}

\newcommand{\ck}{\mathcal{K}}

\newcommand{\real}{\mathbb{R}}
\newcommand{\tran}{\mathsf{\scriptsize T}}

\newcommand{\e}{\mathbb{E}}
\newcommand{\bsa}{\boldsymbol{a}}
\newcommand{\bsx}{\boldsymbol{x}}
\newcommand{\bsz}{\boldsymbol{z}}
\newcommand{\bsone}{\boldsymbol{1}}
\newcommand{\bszero}{\boldsymbol{0}}

\newcommand{\simiid}{\stackrel{\mathrm{iid}}{\sim}}
\newcommand{\toas}{\stackrel{\mathrm{a.s.}}{\to}}

\newcommand{\dunif}{\mathbb{U}}
\newcommand{\dnorm}{\mathcal{N}}

\newcommand{\giv}{\!\mid\!} % 

\newcommand{\prpl}{\text{PrPl}}

\newcommand{\prpleb}{\text{PrPl-EB}}

\newcommand{\eb}{\mathrm{E}}
\newcommand{\ben}{\mathrm{Ben}}

\newcommand{\mc}{\mathrm{MC}}
\newcommand{\rqmc}{\mathrm{RQMC}}
\newcommand{\hk}{\mathrm{HK}}

\newcommand{\rd}{\,\mathrm{d}}
\newcommand{\phz}{\phantom{0}}

\newcommand{\jmp}{\mathrm{jump}}
\newcommand{\knk}{\mathrm{kink}}
\newcommand{\smo}{\mathrm{smooth}}
\newcommand{\fin}{\mathrm{finance}}

\title{Empirical Bernstein and betting confidence intervals for randomized quasi-Monte Carlo}
\date{January 2026}
\author{Aadit Jain\\University of California, San Diego
 \and
 Fred J. Hickernell \\Illinois Institute of Technology
 \and
 Art B. Owen \\ Stanford University
 \and
 Aleksei G. Sorokin \\ University of Chicago
}
\begin{document}
\doublespacing

\maketitle

\begin{abstract}
Randomized quasi-Monte Carlo (RQMC) methods estimate the mean of a random
variable by sampling an integrand
at $n$ equidistributed points. For scrambled digital nets, 
the resulting variance is typically $\tilde O(n^{-\theta})$ 
where $\theta\in[1,3]$ depends on the smoothness of the
integrand and $\tilde O$ neglects logarithmic factors.
While RQMC can be far more accurate than plain Monte Carlo (MC)
it remains difficult to get confidence intervals on RQMC
estimates.  We investigate some empirical Bernstein confidence intervals (EBCI)
and hedged betting confidence intervals (HBCI), both from Waudby-Smith and Ramdas (2024),
when the random variable of interest is subject to known bounds.
When there are $N$ integrand evaluations
partitioned into $R$ independent replicates of $n=N/R$
RQMC points, and the RQMC variance is $\Theta(n^{-\theta})$,
then an oracle minimizing the width of a Bennett
confidence interval would choose $n =\Theta(N^{1/(\theta+1)})$. 
The resulting intervals have
a width that is $\Theta(N^{-\theta/(\theta+1)})$. 
Our empirical investigations had optimal values of
$n$ grow slowly with $N$,
HBCI intervals that were usually narrower than the EBCI
ones, and optimal values of $n$ for HBCI that were equal to or
smaller than the ones for the oracle. 
\end{abstract}

\section{Introduction}

We study the combination of randomized quasi-Monte
Carlo (RQMC) integration to estimate an expectation
with some non-asymptotic (finite sample valid) methods
to get a confidence interval for such an expectation.
RQMC has significant accuracy  benefits over plain Monte Carlo (that
we outline below), but the usual confidence intervals based on
RQMC estimates have justifications that are asymptotic
in their sample size \cite{naka:tuff:2024}. %more fitting than \cite{LEcEtal24a} which we had.
Empirical Bernstein and some related betting
methods (described below) give finite sample assurances for bounded integrands.  Specifically, if it is known that the integrand values must lie
between $0$ and $1$, then those methods can construct a confidence
interval that has at least $1-\alpha$ probability to contain
the integral's value.  This confidence level holds for any integrand
satisfying the given bounds.  

The context for our work is as follows.
Many scientific problems require the numerical 
evaluation of multidimensional
integrals. When the dimension is high enough, then classical tensor product integration rules 
like those in \cite{DavRab84} become too expensive to use.
It is then
common to use plain Monte Carlo (MC) methods instead.  MC
methods based on random sampling typically have root mean
squared errors (RMSEs) of $O(n^{-1/2})$ given $n$
function evaluations.  This rate is slow but it is the
same in any dimension.  MC also allows one to use statistical
methods to quantify the uncertainty in the estimated
integral.  Ordinary quasi-Monte Carlo (QMC) 
methods (e.g.,  \cite{DicPil10a} and \cite{Nie92}) are deterministic
sampling strategies that under a bounded variation assumption
produce integral estimates with error $\tilde O(n^{-1})$ where
$\tilde O(\cdot)$ means we neglect powers of $\log(n)$.
Because plain QMC methods are deterministic, they lose the
uncertainty quantification advantage of MC. RQMC methods, 
surveyed in \cite{LEcLem02a}, allow one
to make independent replicates of a statistically
unbiased QMC method to support variance estimates and
asymptotic confidence intervals.  RQMC methods can also
improve on the convergence rate of QMC methods obtaining
an RMSE of $\tilde O(n^{-3/2})$ under a smoothness assumption
\cite{Owe97,Owe08a} on the integrand.  

In addition to asymptotic confidence interval methods for RQMC,
some other papers have provided finite sample uncertainty quantifications
based on additional assumptions. For purely independent and identically distributed (IID) sampling, knowing a bound on the kurtosis facilitates a finite sample confidence interval for the mean of a random variable in terms of the sample mean and the sample standard deviation \cite{HicEtal14a}.  For deterministic QMC sampling of $\bsx \sim \mathbb{U}[0,1]^d$ with lattices or digital nets, assuming that the coefficients in a suitable orthogonal decomposition of $f : [0,1]^d \to \reals$ decay reasonably, one may derive a deterministic error bound on $\mu  = \mathbb{E}[f(\bsx)]$ \cite{HicJim16a,HicEtal17a,JimHic16a,JagSor23a}.  A different approach is to assume that $f$ is a realization of a Gaussian process and construct a credible interval for $\mathbb{E}[f(\bsx)]$.  If the covariance kernel and the sampling nodes are well-matched, the computational effort required is nearly $O(n)$ and not the $O(n^3)$ effort typically required for such credible intervals \cite{RatHic19a,JagHic22a}.  The GAIL \cite{Gail_ug} and QMCPy \cite{QMCPy2020a} libraries implement  these finite sample uncertainty quantifications under the assumptions outlined above.
%The GAIL \cite{Gail_ug} and QMCPy \cite{QMCPy2020a} libraries provide finite sample uncertainty quantification, so long as one can
%provide some outside information on the integrand.  
%This could be an assumption about a kurtosis quantity.
%These papers \cite{HicJim16a,JimHic16a} use
%assumptions about the way that coefficients in
%a suitable orthogonal decomposition of the integrand decay
%to compute intervals around a QMC estimate that contain the integral.

Here, we study RQMC confidence intervals
assuming that the integrand is bounded between
zero and one.  The advantage of this assumption is that
there are more settings where we are certain that it holds.
%We can then get finite sample confidence
%intervals with coverage of at least 95\% or some other desired level.

The non-asymptotic confidence intervals that we emphasize are the
recently developed hedged betting confidence intervals (HBCI) and some related predictable plug-in empirical Bernstein
confidence intervals (EBCI), both from \cite{WauRam24a}. They are derived
from some infinite confidence sequences and we believe that
they provide the narrowest confidence intervals among
currently available methods.

Section~\ref{sec:notation} gives our notation
and some definitions and background on quasi-Monte
Carlo, randomized quasi-Monte Carlo, and the EBCI and HBCI of \cite{WauRam24a}. In Section \ref{sec:asymptotic}
we show how splitting $N$ observations into $R=N/n$ independent
replicates of $n$ RQMC point sets can give narrower EBCI 
than MC does for an oracle
that knows the RQMC variance at each $n$.  The extent of the
narrowing depends on the RQMC convergence rate which
varies from problem to problem. Perhaps surprisingly,
the more effective RQMC is, the smaller the optimal value
of $n$ becomes.  For an RQMC variance  proportional to $n^{-\theta}$,
taking $n$ proportional to $N^{1/(\theta+1)}$
gives the narrowest intervals.  
The resulting interval widths are
$\Theta(N^{-\theta/(\theta+1)})$ which can be much wider than the 
assumed RQMC
standard deviation of $\Theta(N^{-\theta/2})$.
Section~\ref{sec:finite} makes
some computational investigations.  As predicted by the theory
the best values of $n$ to use grow slowly with $N$.
The empirically best values of $n$ for HBCI are no larger
than the oracle values, and sometimes smaller.
Section~\ref{sec:discussion} has some final comments
and discussion.

\section{Background and notation}\label{sec:notation}

We begin with some notation.
We use $1\{E\}$ to denote a quantity that equals $1$
when expression $E$ holds and is zero otherwise.
The vectors $\bszero$ and $\bsone$ have all components
equal to $0$ and $1$ respectively with a dimension
given by context. We write $a_n = O(b_n)$ as $n\to\infty$ if there exists $C>0$ and $n_0<\infty$ with $|a_n|\le Cb_n$ whenever $n\ge n_0$.
We write $a_n =\Omega(b_n)$ if there exists $C>0$ and $n_0<\infty$ with $a_n\ge Cb_n$ whenever $n\ge n_0$.
We write $a_n = \Theta(b_n)$ if $a_n = O(b_n)$ and $b_n = O(a_n)$.

\subsection{Intervals and martingales}

For $0\le\alpha\le 1$, a $1-\alpha$ confidence interval for a 
parameter $\mu$ is a pair of random quantities $A$ and $B$ with
\begin{align}\label{eq:ci}
\Pr( A\le \mu\le B)\ge 1-\alpha.
\end{align}
A confidence interval is strict when the sides of~\eqref{eq:ci} are equal.
An asymptotic confidence interval is a sequence $(A_n,B_n)$
of random quantities for which
$$
\lim_{n\to\infty}\Pr( A_n\le \mu\le B_n)\ge 1-\alpha.
$$
It is common for asymptotic confidence intervals to
be strict, meaning that this limit equals $1-\alpha$.
The best known example is Student's $t$ confidence
interval for the mean of a distribution with finite
variance, based on a random sample 
from that distribution and justified by the
central limit theorem.  %It is asymptotically strict.
%Fred suggests a paragraph break

The random sequence $(A_n,B_n)$ is a $1-\alpha$ confidence sequence
for $\mu$ if
\begin{align}\label{eq:cs}
\Pr( A_n \le \mu \le B_n,\ \forall n\ge1)\ge 1-\alpha.
\end{align}
Confidence sequences allow us to select a stopping time
$\nu$ based on $\{(A_n,B_n)\mid n\le\nu\}$ and have
at least $1-\alpha$ confidence that $A_\nu\le \mu\le B_\nu$.
We can take the limits to be $\max_{n\le \nu}A_n$
and $\min_{n\le\nu}B_n$.

Martingale theory is one of the main tools for
obtaining confidence sequences.
If random variables $S_i$ for $i\ge1$
satisfy $\e(S_n\giv S_1,\dots,S_{n-1})=S_{n-1}$
then $(S_i)_{i\ge1}$ is a martingale.
If $\e(S_n\giv S_1,\dots,S_{n-1})\le S_{n-1}$
then $(S_i)_{i\ge1}$ is a supermartingale.
Ville's inequality is that for any $\eta>0$
and any nonnegative supermartingale $(S_i)_{i\ge1}$
$$
\Pr\biggl( \sup_{n\ge1}S_n\ge \eta\biggr)\le\frac{\e(S_1)}\eta.
$$
The confidence sequences from \cite{WauRam24a} are derived
by applying Ville's inequality \cite{vill:1939} to nonnegative supermartingales.

The problem we consider is to approximate the finite-dimensional
integral
$$\mu = \int_{[0,1]^d}f(\bsx)\rd\bsx 
=\e(f(\bsx)),\quad\text{for $\bsx\sim\dunif[0,1]^d$}.$$
In MC sampling we take $\bsx_i\simiid \dunif[0,1]^d$ and
estimate $\mu$ by
$$
\hat\mu = \frac1n\sum_{i=1}^nf(\bsx_i).
$$
Let $\sigma^2=\var(f(\bsx))$ for $\bsx\sim\dunif[0,1]^d$.
We assume that $0<\sigma^2<\infty$. In that case the RMSE
of $\hat\mu$ is $\sigma/\sqrt{n}$. For $n\ge2$,
$$s^2 =\frac1{n-1}\sum_{i=1}^n(f(\bsx_i)-\hat\mu)^2$$
is an unbiased estimate of $\sigma^2$. Furthermore
$$
\lim_{n\to\infty}\Pr\Bigl( \sqrt{n}\frac{\hat\mu-\mu}s\le t^{1-\alpha}_{(n-1)}\Bigr) = 1-\alpha
$$
where $t^{1-\alpha}_{(n-1)}$ is the $1-\alpha$ quantile of
Student's $t$ distribution on $n-1$ degrees of freedom.
Then $\hat\mu \pm st^{1-\alpha/2}/\sqrt{n}$ is an
asymptotic $1-\alpha$ confidence interval for~$\mu$.

\subsection{QMC and RQMC}

QMC sampling replaces random points $\bsx_i$ by 
deterministic points that more evenly sample the
unit cube.  There are many ways to quantify that 
property. The most elementary one is the star discrepancy:
$$
D_n^* = D_n^*(\bsx_1,\dots,\bsx_n)=\sup_{\bsa\in[0,1]^d}
\,\biggl| \frac1n\sum_{i=1}^n1\{\bsx_i\in[\bszero,\bsa)\}
-\prod_{j=1}^da_j\biggr|.
$$
Small values of $D_n^*$ show that all anchored boxes $[\bszero,\bsa)$
have very nearly the desired proportion of the $n$ points,
which is the volume of $[\bszero,\bsa)$.
It is possible to choose points $\bsx_1,\dots,\bsx_n$ so
that $D_n^*=\tilde O(n^{-1})$ while plain MC points
have $D_n^*=\tilde O(n^{-1/2})$ \cite{Nie92}.

The improvement in star discrepancy yields an improvement
in integration error via the Koksma-Hlawka inequality \cite{Hic04a}:
\begin{align}\label{eq:koksmahlawka}
|\hat\mu-\mu|\le D_n^*(\bsx_1,\dots,\bsx_n)\times V_{\hk}(f).
\end{align}
As described above $D_n^*$ is a measure of how non-uniform
the points $\bsx_i$ are. The new factor $V_{\hk}(f)$ is the
total variation of $f$ in the sense of Hardy and Krause.
See \cite{Owe05a}. It follows from~\eqref{eq:koksmahlawka} 
that we can get $|\hat\mu-\mu|=\tilde O(n^{-1})$
using QMC. We will refer to this as the BVHK rate.

If we knew $D_n^*$ and $V_\hk(f)$, then~\eqref{eq:koksmahlawka} would provide a
perfect quantification of our uncertainty about $\mu$.
It would be non-asymptotic and even non-probabilistic.
Unfortunately, $D_n^*$ is generally very expensive
to obtain and $V_{\hk}(f)$ is ordinarily far harder
to compute than $\mu$. As a result, equation~\eqref{eq:koksmahlawka}
shows us that MC can be outperformed but it does
not provide a generally usable error estimate.

RQMC methods generate random points $\bsx_i$ that
simultaneously satisfy two properties:
\begin{compactenum}[\quad\bf1)]
\item $\bsx_i\sim\dunif[0,1]^d$, for all $i=1,\dots,n$, and
\item $D_n^*(\bsx_1,\dots,\bsx_n)=\tilde O(n^{-1})$ almost surely.
\end{compactenum}
From the second property, $|\hat\mu-\mu|=\tilde O(n^{-1})$.
As a result the RMSE of RQMC is $\tilde O(n^{-1})$
in the BVHK case.

To get an approximate confidence interval from RQMC
we can form $R$ statistically independent RQMC estimates
$\hat\mu_1,\dots,\hat\mu_R$. Using the first property
above these IID estimates are unbiased and we can get
an asymptotic confidence interval for $\mu$ of the form
$$
\hat\mu \pm S t_{(R-1)}^{1-\alpha/2}/\sqrt{R}
$$
for
$$
\hat\mu = \frac1R\sum_{i=1}^R\hat\mu_i
\quad\text{and}\quad
S^2 = \frac1{R-1}\sum_{i=1}^R(\hat\mu_i-\hat\mu)^2.
$$
Here $S^2/R$ is an unbiased estimate of the MSE
of the pooled estimate $\hat\mu$.
The confidence level $1-\alpha$ is obtained in
the limit as $R\to\infty$ \cite{naka:tuff:2024}
for fixed $n$. That convergence is generally quicker for more nearly
Gaussian $\hat\mu_i$. % We could quote Berry-Esseen but that brings in a third moment.  I think this remark should be enough.

In addition to getting an unbiased estimate of
the RQMC variance, randomization can provide
some other benefits.  First, some RQMC methods
can attain better accuracy than plain QMC.
Under sufficient smoothness, the scrambled  digital sequences
in \cite{Owe95} attain RMSEs of $\tilde O(n^{-3/2})$
\cite{Owe97,Owe08a}. The same holds for the scramblings
of \cite{Mat98}. We will refer to this as the `smooth case'.
Higher order digital nets \cite{Dic11a} can attain even
better convergence rates, though the rates are not
necessarily evident empirically when $d$ is moderately
large, which \cite{nuyens2010higher} attribute to numerical precision.

Those RQMC methods have further
useful properties.  First,  $\var(\hat\mu) = o(1/n)$ as $n\to\infty$
for any integrand with $\sigma^2<\infty$, and so their
efficiency with respect to MC becomes unbounded as $n\to\infty$.
Furthermore, the condition of finite
variation in the sense of Hardy and Krause can be quite
strict. For example, when $d\ge2$, step discontinuities in the integrand
$f$ ordinarily make $V_{\hk}(f)$ infinite unless those discontinuities
are axis parallel \cite{Owe05a}. This makes QMC
unattractive for estimating integrals of binary
valued functions, while RQMC will still attain an
RMSE that is $o(n^{-1/2})$.
%Strictly speaking we need the function to be measurable. But even more strictly, we already assumed that when we assumed that $\mu$ exists.

While there are numerous RQMC methods in use, we
will restrict ourselves here to scramblings of digital nets
and sequences, such as the ones of Sobol' \cite{Sob67}.
Our variance formulas assume that the scramblings 
are either those of \cite{Owe95} or
\cite{Mat98}.  In our computations we used the
linear matrix scramble of \cite{Mat98} with a digital shift.

While RQMC methods yield very good accuracy, most
of the known confidence interval methods for them
supply only asymptotic confidence as $R \to \infty$ \cite{naka:tuff:2024}. 
In some very limited settings
there are central limit theorems for $\hat\mu$
as $n\to\infty$ \cite{Loh01}.
Having the confidence statements be asymptotic in $R$
is unfortunate since the RMSE of $\hat\mu$ 
is proportional to $R^{-1/2}$ while the RMSE
vanishes at a faster rate in $n$.

\subsection{Hoeffding and Empirical Bernstein intervals}

We would like a confidence interval that is non-asymptotic,
meaning that the coverage guarantee~\eqref{eq:ci} holds for  limits $A$ and $B$ computed from a sample of size $n$.
This can often be obtained under parametric statistical
models in which the distribution of $f(\bsx)$  belongs
to a known finite dimensional family such as the Gaussian
distributions.  It is not desirable to make such assumptions
and so we prefer a nonparametric method with non-asymptotic
confidence.  

There is a theorem of Bahadur and Savage \cite{BahSav56} that
reveals sharp constraints on our ability to construct
nonparametric and non-asymptotic confidence
intervals.  Here is the description from \cite{err4qmc}:
% I just now got and fixed the galleys in that article so the following verbatim quote makes sense
\begin{quote}They consider a set $\cf$ of distributions on $\real$.
Letting $\mu(F)=\e(Y)$ for $Y\sim F\in\cf$, their conditions are:
\begin{compactenum}[\quad (i)]
\item For all $F\in\cf$, $\mu(F)$ exists and is finite.
\item For all $m\in\real$ there is $F\in\cf$ with $\mu(F)=m$.
\item $\cf$ is convex: if $F,G\in\cf$ and $0<\pi<1$,
  then $\pi F+(1-\pi)G\in \cf$.
\end{compactenum}
Then their Corollary 2 shows that a Borel set constructed
based on $Y_1,\dots,Y_N\stackrel{\mathrm{iid}}
\sim F$
that contains $\mu(F)$ with probability at least $1-\alpha$
also contains any other $m\in\real$ with probability at least $1-\alpha$.
More precisely: we can get a confidence set, but not a useful one.
They allow $N$ to be random so long as $\Pr(N<\infty)=1$.
\end{quote}

We can escape the restriction from Bahadur and
Savage by stipulating that every distribution in $\cf$
has support in a known interval $[a,b]$ of finite length.
This violates their clause (ii).
After a linear transformation we take that interval
to be $[0,1]$.

We will use $Y_i = \hat\mu_i$ for $i=1,\dots,R$
to represent the values of the $R$ replicated
RQMC points that go into our confidence interval
calculation.
We set $\bar Y_R = (1/R)\sum_{i=1}^RY_i$.  
We assume that $Y_i$ are IID with the same
distribution as $Y$ and that
$0\le Y\le1$ holds with probability one.
Then
\begin{align}\label{eq:hoeffding}
\bar Y_R \pm \sqrt{\frac{\log(2/\alpha)}{2R}}
\end{align}
for $0<\alpha<1$ is a $1-\alpha$ confidence interval for $\mu=\e(Y_i)$.  Equation~\eqref{eq:hoeffding} gives
the confidence interval of Hoeffding \cite{hoef:1963}.  The interval is not always
nested within $[0,1]$, but we can simply take its
intersection with $[0,1]$. That will preserve the confidence level.

While Hoeffding's interval is a valid confidence interval
for the mean of any distribution supported on $[0,1]$,
it can be quite conservative for some
of them, giving confidence over $1-\alpha$. This raises
the possibility of getting coverage $1-\alpha$ from
narrower intervals.  
One approach is to make use of 
$\sigma^2=\var(Y)$ if it is known. 
Hoeffding's confidence interval is a consequence of
Hoeffding's exponential probability bound in \cite{hoef:1963}.
Other exponential probability bounds take advantage of
a known value $\sigma^2=\var(Y)$ or even a sample estimate
of that variance.  Those bounds can be used to get
improved confidence intervals.

Theorem 3 of Maurer and Pontil \cite{maurer2009empirical}
is that
$$
\e(Y) - \frac1R\sum_{i=1}^RY_i \le \sqrt{\frac{2\sigma^2\log(1/\delta)}R}
+ \frac{\log(1/\delta)}{3R}.
  $$
holds with probability at least $1-\delta$.
They call this Bennett's inequality, because it derives
from an inequality in \cite{benn:1962}. 
They do not supply a proof, but it can be proved from Bernstein's
inequality given as Theorem 3 of \cite{bouc:lugo:bous:2003}.
Taking $\delta =\alpha/2$ we get
\begin{align}\label{eq:bennettci}
\bar Y_R \pm \sqrt{\frac{2\sigma^2\log(2/\alpha)}R}
+ \frac{\log(2/\alpha)}{3R}
\end{align}
as a $1-\alpha$ confidence interval for $\mu$.
We will make use of the Bennett confidence
intervals~\eqref{eq:bennettci} for some
theoretical investigations. 

In most practical settings where $\mu$ is unknown,
$\sigma^2$ will also be unknown. Replacing $\sigma^2$
by $s^2$ in~\eqref{eq:bennettci} would not always
give a valid confidence interval.  
Theorem 4 of \cite{maurer2009empirical} shows that
for $R\ge2$
\begin{align}\label{eq:eb}
    \bar Y_R\pm
\sqrt{\frac{2s^2\log(4/\alpha)}R}
+ \frac{7\log(4/\alpha)}{3(R-1)}.
\end{align}
is a $1-\alpha$ confidence interval for $\mu$.
This is called an empirical Bernstein confidence
interval named after an inequality of Bernstein
that is similar to Bennett's.  We see that both terms
on the right hand side of~\eqref{eq:eb} are larger
than the ones in the Bennett interval~\eqref{eq:bennettci}.
Their interval has improved asymptotic constants compared to the original
empirical Bernstein inequality from
\cite{audi:muno:szep:2007}. The EBCI and HBCI
from \cite{WauRam24a} provide a further asymptotic improvement.

\subsection{Betting methods}

Our interest in this problem was sparked in part by some recent
work by \cite{WauRam24a} on confidence intervals generated
by betting arguments.  They consider a more general setting
of confidence sequences, as in equation \eqref{eq:cs}, 
that can then be specialized to confidence intervals as needed.  
For random $Y_i\in[0,1]$ they assume
that $\e(Y_i\giv Y_1,\dots,Y_{i-1})=\mu$ for all $i\ge1$ 
without assuming
that the $Y_i$ are identically distributed or even that
they are independent.
For $i=1$, the conditional expectation 
above is simply $\e(Y_1)=\mu$.

Suppose that the conditional mean is really $\mu$ and 
there is a null hypothesis $H_0(m)$ that
$\e(Y_i\giv Y_1,\dots, Y_{i-1})=m$. Then if $m\ne\mu$
somebody starting with a stake of \$1.00 can
make a series of bets against $H_0(m)$
as the $Y_i$ are revealed in order
and have an expected fortune that grows without bound.
Just prior to time $i$, the bettor picks a value
$\lambda_i = \lambda_i(m)\in(-1/(1-m),1/m)$. 
% I wrote in lambda_i without the m as a synonym in case we use it below in more places.
The bettor's capital at time $t\ge1$ is
$$
\ck_t(m)= \prod_{i=1}^t\bigl(1+\lambda_i(m)(Y_i-m)\bigr).
$$
Taking $\lambda_i(m)>0$ amounts to betting that $Y_i>m$. Then their capital is multiplied by $1+\lambda_i(Y_i-m)$. If they are
right then their capital grows, but if $Y_i<m$, then it shrinks.
To take $\lambda_i(m)<0$ is to bet that $Y_i<m$.

Section 4 of \cite{WauRam24a} notes that $\ck_t(\mu)$ is
a nonnegative martingale, so that $\e(\ck_t(\mu))=1$ for all $t\ge1$.
There is no way for the bettor to pick bet sizes $\lambda_i(\mu)$
to make an expected profit.   Then from Ville's Theorem
$$
\Pr( \ck_t(\mu)\le 1/\alpha,\ \forall t\ge1)\ge1-\alpha.
$$
If we watch $\ck_t(\mu)$ indefinitely, then
there is at most probability $\alpha$ that it will
ever go above $1/\alpha$.  As a result we can get
a confidence sequence by retaining at time $t$,
all the values $m$ for 
which $\max_{1\le\tau\le t}\ck_\tau(m)\le1/\alpha$. 
This idiom uses a hypothetically infinite ensemble of 
bettors to do this, but \cite{WauRam24a} give
implementable algorithms for it.

The other side of the coin is that there are good betting
strategies when $m\ne\mu$.  Under those, the bettor's expected fortune
will grow steadily when $m\ne\mu$. That is what one needs
to make a confidence sequence not just valid but also useful
in having a width that converges to zero with increased sampling.

One version of the betting strategy above produces confidence
sequences that are analogous to Hoeffding intervals
and another produces empirical Bernstein intervals.  We present
the predictable plug-in EBCI followed by a hedged betting interval that
combines two betting strategies.
We do not include the predictable plug-in Hoeffding intervals. Those
do not use a sample variance, and they are 
wider than the others for larger $t$.

The confidence intervals we study make use of running sample moments
$$
\hat\mu_t = \frac{1/2+\sum_{i=1}^tY_i}{t+1}
\quad\text{and}\quad
\hat\sigma^2_t = \frac{1/4+\sum_{i=1}^t(Y_i-\hat\mu_i)^2}{t+1}
$$
with the means and standard deviations 
both biased slightly towards $1/2$.
In their `predictable plug-in empirical Bernstein'
intervals, the betting amounts are
$$
\lambda_t^\prpleb  = \sqrt{\frac{2\log(2/\alpha)}{\hat\sigma^2_{t-1}t\log(1+t)}}\,\wedge c
$$
for some $c\in(0,1)$ with $c=1/2$ or $3/4$
given as reasonable defaults. 
Note that this method only bets that $Y_i>m$.
Theorem 2 of \cite{WauRam24a} 
gives a $1-\alpha$ confidence sequence
\begin{align}\label{eq:purpleinterval}
C_t^\prpleb = \frac{\sum_{i=1}^t\lambda_iY_i}{\sum_{i=1}^t\lambda_i}
\pm \frac{\log(2/\alpha)+\sum_{i=1}^t\nu_i\psi_{\eb}(\lambda_i)}{\sum_{i=1}^t\lambda_i},
\end{align}
where
$$
\nu_i = 4(Y_i-\hat\mu_{i-1})^2
\quad\text{and}\quad \psi_{\eb}(\lambda)
= (-\log(1-\lambda)-\lambda)/4.
$$
The running intersection
$\cap_{1\le\tau\le t}C^\prpleb_\tau$
is also a $1-\alpha$ confidence sequence for~$\mu$ and
of course we can intersect any of these intervals with $[0,1]$.

These confidence sequences can be specialized to confidence
intervals when we have a fixed target sample size $R$
in mind.  For that case \cite{WauRam24a}
recommends
\begin{align}\label{eq:purplelambda}
\lambda_i^{\prpleb(R)} = \sqrt{\frac{2\log(2/\alpha)}{R\hat\sigma^2_{i-1}}}\,\wedge c,\quad i=1,\dots,R.
\end{align}
%Later $R$ will be a number of RQMC replicates.
The above formula for $\lambda_i$ uses the ordering
of the data values and puts unequal weight on the
$R$ different $Y_i$. There is a permutation strategy
in \cite{WauRam24a} to treat the $Y_i$ values more
symmetrically, but they report that it makes little difference.

The EBCI we study take the form
$$
\bigcap_{1\le\tau\le t}C_t^{\prpleb(R)}
$$
where $C_t^{\prpleb(R)}$ is the interval in~\eqref{eq:purpleinterval}
using $\lambda_i$ from~\eqref{eq:purplelambda}.
Equation (17) of \cite{WauRam24a} gives the
scaled half width of these EBCI as
$$
\sqrt{R}\biggl(\frac{\log(2/\alpha)+\sum_{i=1}^R\nu_i\psi_\eb(\lambda_i)}{\sum_{i=1}^R\lambda_i} \biggr)\toas \sigma\sqrt{2\log(2/\alpha)}
$$
as $R\to\infty$.  The corresponding limit for the empirical
Bernstein CIs of \cite{maurer2009empirical}
is $\sigma\sqrt{2\log(4/\alpha)}$.

The hedged betting process of \cite{WacMenSch08} works as follows.
For positive predictable sequences $\lambda_i^+(m)$ and $\lambda_i^-(m)$ defined below, they use capital processes
$$
K_t^+(m) = \prod_{i=1}^t\bigl(1+\lambda_i^+(m)(Y_i-m)\bigr)\quad\text{and}\quad
K_t^-(m) = \prod_{i=1}^t\bigl(1-\lambda_i^+(m)(Y_i-m)\bigr)
% ABO: we were missing right parens.  I make the outer ones a touch bigger too.
$$
that bet, respectively, on $Y_i>m$ and $Y_i<m$.  The hedged process is
$$
K_t^\pm(m) = \max\bigl(\theta K_t^+(m),(1-\theta)K_t^-(m)\bigr)
$$
for $0\le\theta\le1$. For any $\theta$, $K_t^\pm(\mu)$ is a nonnegative martingale
hence subject to Ville's theorem.  The supplement to \cite{WauRam24a}
suggests $\theta=1/2$ as a default.  If $m\ne\mu$ then $K_t^\pm(m)$ will tend to grow with $t$.
Of course the hedging $\theta$ has nothing to do with the RQMC variance
rate $\theta$ and the quantities are distinct enough that no confusion
should arise.

The hedged betting process produces confidence sequences and, as they did
for the empirical
Bernstein processes, they recommend a way to produce a confidence interval
based on $R$ sample values.  For that
they take
$$\lambda_i^+(m)=|\tilde\lambda_i^\pm|\wedge\frac{c}m
\quad\text{and}\quad
\lambda_i^-(m)=|\tilde\lambda_i^\pm|\wedge\frac{c}{1-m}
$$
for $i=1,\dots,R$ where 
$$\tilde\lambda_i^\pm=\sqrt{\frac{2\log(2/\alpha)}{R\hat\sigma^2_{i-1}}}$$
for $c\in[0,1)$, with suggested defaults of $c=1/2$ or $c=3/4$.
The confidence interval is then
$$\bigl\{m\in[0,1]\mid \max_{1\le i\le R}K_i^\pm(m)\le1/\alpha\bigr\}.$$

\section{Asymptotic confidence interval comparisons}\label{sec:asymptotic}
Our confidence intervals will use $R$ IID replicates that each
use $n$ RQMC points in $[0,1]^d$.  That takes $N=nR$
evaluations of $f$. 
Here we study how to get the narrowest confidence
intervals with that budget of $N$ function evaluations.

We take 
$$Y_i=\frac1n\sum_{k=1}^nf(\bsx_{i,k})$$
where for each $i$, $\bsx_{i,1},\dots,\bsx_{i,n}$ is an RQMC
point set. The $Y_i$ are IID when we use independent
scrambles for all $R$ point sets. Let $\sigma^2_n = \var(Y_i)$.
We know that $\sigma^2_n = o(n^{-1}).$
For the BVHK case, $\sigma^2_n=\tilde O(n^{-2})$.  
For the smooth case,
$\sigma^2_n=\tilde O(n^{-3})$.

We know from \cite{WauRam24a} that the widths of their
predictable plug-in EBCI approaches
that of some Bennett intervals described below, as $R\to\infty$
for fixed $n$.
The same is true of their HBCI. Here we
study the value of $n$ that optimizes the Bennett intervals.

As mentioned above, the tilde in $\tilde O$ hides powers of $\log(n)$.
Those powers are present in error bounds for adversarially
chosen integrands but do not seem to come up in
real problems \cite{schl:2002,wherearethelogs}.
We will consider working models of the form
\begin{align}\label{eq:themodel}
\sigma^2_n = \sigma^2_0n^{-\theta},
\end{align}
 even though the actual variance is not
necessarily any power of $n$. This approach is
less cumbersome than bounding the variance between
$\Omega(n^{-\theta})$ and $O(n^{-\theta+\epsilon})$,
and we believe it provides clearer insights.

We know that taking $\theta=1$ underestimates
the quality of RQMC because the variance
is $o(1/n)$ for the RQMC methods we use. 
The smooth case with $\theta = 3$ is optimistic
while the BVHK case with $\theta=2$ is intermediate.

If $R\to\infty$ for fixed $n$, then the predictable
plug-in EBCI
have half widths $H^\prpleb$ that satisfy
\begin{align}\label{eq:purplewidth}
\sqrt{R}H^{\prpleb(R)}\to \sigma_n\sqrt{2\log(2/\alpha)}.
\end{align}
The HBCI have half widths
$H^\prpl$ that satisfy the same rate.
As a result, both methods
give interval widths that are $O(n^{-\theta/2}R^{-1/2})$
as $R\to\infty$.
For a fixed product $N=nR$, this width is narrowest at $R=1$
and $n=N$, but of course that is infeasible and also
that argument ignores the fact that~\eqref{eq:purplewidth}
is based on a limiting argument as $R\to\infty$.

To study the tradeoff between $n$ and $R$, we use Bennett's 
inequality which gives a half width of
\begin{align}\label{eq:bennettwidth}
H^\ben(n,R)=\sqrt{\frac{2\sigma_n^2\log(2/\alpha)}R}
+ \frac{\log(2/\alpha)}{3R}
\end{align}
when using $R$ IID observations that each have variance
$\sigma^2_n$. An oracle that knew $\sigma^2_n$ could
use this formula to select the best $n$ and $R$
for a confidence interval subject to a constraint $nR=N$.
The predictable plug-in EBCI
and the HBCI
from \cite{WauRam24a} do not assume a known variance
for the $Y_i$.
However those intervals attain the same asymptotic limiting width
as $R\to\infty$ that an oracle would get from Bennett's formula. 

Next, we investigate the oracle's choice of $n$ and $R$
to minimize the half width subject to $nR=N$. 
While $n$ and $R$ both have to be positive integers
with product $N$, we will first relax the problem
to a continuous variable $n$ with $R=N/n$.  After that, we
discuss integer solutions.

\begin{theorem}\label{thm:goodn}
Let $\sigma^2_n$ follow~\eqref{eq:themodel} and choose $N>0$.
If $\theta=1$, then the minimizer of~\eqref{eq:bennettwidth} 
over $n\in[1,\infty)$ is $n_*=1$.
If $\theta>1$, then the minimizer of~\eqref{eq:bennettwidth}
over $n\in(0,\infty)$ is 
\begin{align}\label{eq:goodn}
n_*& 
%= \left(\frac{3(\theta-1)\sigma_0
%\sqrt{N}}{\sqrt{2\log(2/\alpha)}}\right)^{2/(\theta+1)}
 = 
\left(\frac{9(\theta-1)^2\,\sigma_0^2
N}{2\log(2/\alpha)}\right)^{1/(\theta+1)}.
\end{align}\end{theorem}

\begin{proof}
Under the model~\eqref{eq:themodel}, with $R=N/n$
the oracle's half width is
\begin{align}\label{eq:benh}
H(n) = H^\ben(n,N/n)=\sigma_0n^{(1-\theta)/2}N^{-1/2}
\sqrt{2\log(2/\alpha)}
+\frac{\log(2/\alpha)}{3}\frac{n}N.
\end{align}
If $\theta=1$, then $H(n)$ is minimized over $[1,\infty)$ at $n=1$.

For $\theta>1$, the function $H(n)$ is a convex function of $n$.
It has a unique minimum over $(0,\infty)$ at
some $n_*$ where $H'(n_*)=0$. Now
\begin{align*}
H'(n)=\frac{1-\theta}2\frac{\sigma_0}{\sqrt{N}}n^{-(\theta+1)/2}
\sqrt{2\log(2/\alpha)}  + \frac{\log(2/\alpha)}{3N}.
\end{align*}
This vanishes at 
\begin{align*}
n^{-(\theta+1)/2} = 
\frac{2\log(2/\alpha)/(3N)}
{(\theta-1)\sigma_0
\sqrt{2\log(2/\alpha)/N}}=
\frac{\sqrt{2\log(2/\alpha)}}
{3(\theta-1)\sigma_0
\sqrt{N}}
\end{align*}
from which~\eqref{eq:goodn} follows.
\end{proof}

If $n_*<1$, then $n=1$ is best.
This includes the degenerate case with $\sigma_0=0$.
If $n_*>N$, then $n=N$ is best.
If $1<n_*<N$ and $n_*$ is not an integer then the
best integer $n$ is either $\lfloor n_*\rfloor$
or $\lceil n_*\rceil$ by the convexity noted in the proof.

For scrambled Sobol' points it is generally best to
take $n$ to be a power of 2 \cite{Owe22a}, which is not
reflected in Theorem~\ref{thm:goodn}. We incorporate that restriction
into some numerical examples in Section~\ref{sec:finite}.

As noted above, when $\theta = 1$, the best $n$ is $n=1$. In that case RQMC is the same as MC, using one uniformly distributed point in $[0,1]^d$. This result holds for any sampler that gets the ordinary MC rate, even if it has a favorable constant.  In particular, any MC variance reduction techniques that do not change the convergence rate will not lead to using $n>1$.  They will improve the asymptotic confidence interval width by reducing $\sigma_0$.  

When $\theta>1$, we see that the shortest intervals
come from taking $n = \Theta(N^{1/(\theta+1)})$.
The BVHK rate ($\theta=2$) is then $n=\Theta(N^{1/3})$ while the
rate for smooth integrands ($\theta=3$) is $n=\Theta(N^{1/4})$.
We see that when the RQMC convergence rate becomes better,
the optimal RQMC sample size to use grows at a slower
rate in $N$.  This may seem like a paradox, but the explanation
can be seen in equation~\eqref{eq:bennettwidth}.
If $\sigma_n$ drops very quickly then the $O(1/\sqrt{R})$ term
%NB: that one is O() not Theta() because it might have sigma_0=0
% For now I don't want have O( ) and Theta( ) in the same phrase
it is in does not dominate the other $O(1/R)$ term as much
and a larger $R$ is better.

Because the oracle takes $n$ to be a smaller
power of $N$ for larger $\theta$, we might
wonder why it always picks $n=1$ for the smallest value, $\theta=1$. 
This stems from the factor $\theta-1$ in the constant of proportionality
in \eqref{eq:goodn}.
For fixed $\sigma_0$, $N$ and $\alpha$ we could 
maximize~\eqref{eq:goodn} over $\theta$
to see which rate gives the largest $n$.

What we see from Theorem~\ref{thm:goodn} is that the optimal
$n$ grows slowly with $N$. The variance of the average of $R$
RQMC samples using $n=N/R$ is
$$
\Theta(R^{-1}n^{-\theta}) = \Theta( N^{-1}n^{1-\theta})
= \Theta( N^{-1}N^{(1-\theta)/(\theta+1)})
%= \Theta( N^{-1+(1-\theta)/(\theta+1)})
= \Theta( N^{-2\theta/(\theta+1)}).
$$
By requiring a guaranteed confidence level, rather than an asymptotic one, we
get a higher variance for $\hat\mu$ than $\Theta(N^{-\theta})$
that we could have had with $R=O(1)$.
For $\theta=2$ the variance rises
from $\Theta(N^{-2})$ to $\Theta(N^{-4/3})$
and for $\theta=3$, it rises from
$\Theta(N^{-3})$ to $\Theta(N^{-3/2})$.
Put another way, obtaining a nonasymptotic confidence
level raises the variance by $\Theta(N^{2/3})$
and $\Theta(N^{3/2})$ in those two cases.

While we cannot always know the best value of $n$ to use
in a given setting, Theorem~\ref{thm:goodn} gives us some guidance.
The asymptotic value of $\theta$ may be known from theory.
We may also have experience with similar integrands to make a
judgment about what value of $\theta$ is reasonable in a given
setting, or we can do some preliminary sampling to get a working
value of $\theta$ before constructing a betting confidence interval.
For any choice of $\theta$, that interval will be valid, though possibly 
wider than an interval using the true optimal $n$.
Because $0\le Y_i\le1$, we know that $\sigma_0^2\le 1/4$. Therefore
under the model~\eqref{eq:themodel}, Theorem~\ref{thm:goodn}
provides the guidance to take
\begin{align}\label{eq:guidance}
n\le \Bigl(\frac{9(\theta-1)^2N}{8\log(2/\alpha)}\Bigr)^{1/(\theta+1)}.
\end{align}
For example, with $N=2^{10}$ and $\alpha=0.05$, the bound in~\eqref{eq:guidance}
is always below $7$ whenever $1\le\theta\le3$. The best integer
value is $7$ for some values of $\theta$.
For higher confidence levels (smaller $\alpha$), the optimal
$n$ is smaller still. 

The values of $N$ and $\alpha$ enter the Bennett half width
of~\eqref{eq:benh} only through  the expression 
$$N_\alpha\equiv \frac{N}{\log(2/\alpha)}.$$ 
Thus we can think of smaller $\alpha$ 
as providing a smaller effective budget $N_\alpha$.

\begin{theorem}
Under the conditions of Theorem~\ref{thm:goodn}
with $\theta>1$ and $\sigma_0>0$, 
let $H_{\mc}$ be the half width $H^\ben(1,N)$ using $n=1$
and $H_{\rqmc}$ be the half width $H^\ben(n,N/n)$ using
$n=n_*$ from~\eqref{eq:goodn}. 
Let $N_\alpha=N/\log(2/\alpha)$.
Then for $\theta>1$
\begin{align*}
\frac{H_\rqmc}{H_\mc} 
&= 
\frac{\theta+1}
{\sqrt{{2\sigma^2_0N_\alpha}}+1/3}
\biggl(\frac12
[{3(\theta-1)}]^{1-\theta}
N_\alpha\sigma_0^2\biggr)^{1/(\theta+1)}\\
&=\Theta( N^{\frac{1-\theta}{2\theta+2}})
\end{align*}
as $N\to\infty$.
\end{theorem}
\begin{proof}
Using equation~\eqref{eq:benh} in the numerator,
the ratio of these widths is
\begin{align*}
\rho=
\frac{H_\rqmc}{H_\mc}
&=
\frac{n^{(1-\theta)/2}\sqrt{2\sigma^2_0\log(2/\alpha)/N}+n\log(2/\alpha)/(3N)}
{\sqrt{{2\sigma^2_0\log(2/\alpha)}/N}+\log(2/\alpha)/(3N)}\\
&=
\frac{n^{(1-\theta)/2}\sqrt{2\sigma^2_0N_\alpha}+n/3}
{\sqrt{{2\sigma^2_0N_\alpha}}+1/3}.
\end{align*}
From equation~\eqref{eq:goodn}
\begin{align*}
n^{(1-\theta)/2}
%&=
%\left(\frac{3(\theta-1)\sigma_0
%\sqrt{N}}{\sqrt{2\log(2/\alpha)}}\right)^{(1-\theta)/(\theta+1)}\\
&=({3(\theta-1)\sigma_0\sqrt{N_\alpha/2}})^{(1-\theta)/(\theta+1)}
\end{align*}
and we may also write
$n=(3(\theta-1)\sigma_0(N_\alpha/2)^{1/2})^{2/(\theta+1)}$.

So for $\theta>1$, the width ratio is
\begin{align*}
\rho&=
\frac{({3(\theta-1)\sigma_0\sqrt{N_\alpha/2}})^{(1-\theta)/(\theta+1)}\sqrt{2\sigma^2_0N_\alpha}
+(3(\theta-1)\sigma_0(N_\alpha/2)^{1/2})^{2/(\theta+1)}/3}
{\sqrt{{2\sigma^2_0N_\alpha}}+1/3}\\
&=\frac{({3(\theta-1)})^{(1-\theta)/(\theta+1)}2^{\theta/(\theta+1)}
+(3(\theta-1))^{2/(\theta+1)}2^{-1/(\theta+1)}/3}
{\sqrt{{2\sigma^2_0N_\alpha}}+1/3}(N_\alpha\sigma_0^2)^{1/(\theta+1)}\\
%&=[{3(\theta-1)}]^{(1-\theta)/(\theta+1)}2^{-1/(\theta+1)}
%\frac{\theta+1}{\sqrt{{2\sigma^2_0N_\alpha}}+1/3}
%(N_\alpha\sigma_0^2)^{1/(\theta+1)}\\
&=\frac{\theta+1}
{\sqrt{{2\sigma^2_0N_\alpha}}+1/3}
\biggl(\frac12
[{3(\theta-1)}]^{1-\theta}
N_\alpha\sigma_0^2\biggr)^{1/(\theta+1)}
\end{align*}
as required.
\end{proof}

For the BVHK case with $\theta=2$, 
we find that RQMC narrows the widths by $\Theta(N^{-1/6})$.
For the smooth case with $\theta=3$,
we get a more favorable width ratio of $\Theta(N^{-1/4})$.
The MC widths are $\Theta(N^{-1/2})$, and so for the
BVHK case the RQMC widths are $\Theta(N^{-2/3})$
while for the smooth case, they are $\Theta(N^{-3/4})$.

\section{Finite sample evaluations}\label{sec:finite}

Here we present some finite sample evaluations
of the EBCI (specifically the predictable plug-in one) and 
the HBCI, both using RQMC, for 95\% confidence. We begin with finite $N$ and for some integrands with known RQMC variance, we find the $n$ that
minimizes the interval widths from Bennett's inequality.
They are then the exact best $n$ for Bennett's inequality which the
HBCIs and EBCIs of \cite{WauRam24a} match asymptotically.
After that we present finite sample results on the
EBCIs and HBCIs for a variety of integrands differing in smoothness and dimension. The HBCIs are usually narrower.
Finally, we compare the empirically optimal $n$ for the HBCIs to the values that optimize Bennett's inequality. The HCBI widths approach those from Bennett's inequality for large $R$ and we find that the optimal $n$ for HBCI tend to be the same or smaller than the ones we get from Bennett's inequality.

For our HBCI computations, we used $\theta=1/2$. The simulations presented here were run at
\href{https://github.com/aaditj1962161/Betting-Paper-Simulations-for-QMC}{Betting-Paper-Simulations-for-QMC}. The software we used comes from \cite{hedged_algo} which has code to reproduce the figures in \cite{WauRam24a}. It calls the HBCI and predictable plug-in EBCI algorithms from \cite{bet_algo}.
The scrambled Sobol' points we used were from QMCPy \cite{QMCPy2020a}.

\subsection{Semi-empirical evaluations}\label{sec:semiempirical}

We can compute the width of intervals based on Bennett's inequality
for some specific cases.  We saw in Section~\ref{sec:asymptotic}
that the RQMC sample sizes $n$ grow at slow rates
$\Theta(N^{1/3})$ or $\Theta(N^{1/4})$ as the budget $N$ increases.
Here we get a more detailed view of that phenomenon using
the asymptotic formula for some specific variance rates.
We can compute these confidence interval widths for very
large values of $N$ that would be impractical to simulate.

The digital nets of Sobol' \cite{Sob67} under the nested
uniform scramble of \cite{Owe95} have a stratification
property.  The $n$ values of $x_{ij}$ for $i=1,\dots,n$
and any component $j=1,\dots,d$ have one value uniformly
distributed in each interval $[(\ell-1)/n,\ell/n)$
for $\ell=1,\dots,n$ and those $n$ values are mutually
independent. This allows us to work out
the finite sample RQMC variances for
some illustrative examples.  The random linear scramble of \cite{Mat98} combined
with a digital shift attains the same variance as the nested
uniform scramble, but it has some dependencies among the stratified
sample values.  In both cases, the $d$ vectors $(x_{1j},
\dots,x_{nj})\in\real^n$ are mutually independent.

We only consider values of $n$ that are powers of $2$
which is a best practice when using Sobol' sequences \cite{Owe22a}.
The smooth error rate $\tilde O(n^{-3/2})$ for the RMSE
is attainable along powers of $2$
but not for general $n$. An elegant argument due
to Sobol' \cite{sobo:1998} explains why.
He notes that
$$
|f(\bsx_{n+1})-\mu|
%= |(n+1)(\hat\mu_{n+1}-\mu)-n(\hat\mu_n-\mu)|
\le (n+1)|\hat\mu_{n+1}-\mu|+n|\hat\mu_n-\mu|.
$$
A rate of $o(n^{-1})$ would imply that $f(\bsx_{n+1})$ itself
consistently estimates $\mu$. For any sampling
algorithm, that could only be true for a very minimal 
set of functions~$f$.
We also take $N$ to be a power of 2.

For $d=1$ and $f$ of bounded variation in the usual sense,
we get $\var(\hat\mu)\le V_{\hk}(f)^2/n^2$.
This holds because $D_n^*$ for the stratified sample
cannot be larger than $1/n$.
For illustration we choose a function with $V_{\hk}(f)=1$
for which we know exactly what the RQMC variance is
for $n$ a power of $2$.
That integrand is $f(x)=1\{x<1/3\}$.
It is constant within $n-1$ of the intervals $[(\ell-1)/n,\ell/n)$.
All of this function's variation happens 
within just one of those intervals.
By inspecting the base $2$ expansion $1/3=0.010101\cdots$
we can see that the remaining value of $f(x_\ell)$ 
is $1$ or $0$ and it equals
$1$ with a probability that is either $1/3$ if $\log_2(n)$
is an odd integer or is $2/3$ if $\log_2(n)$ is an
even integer. It follows that $\var(\hat\mu_n)=2/(9n^2)$.

Table~\ref{tab:bvhkcase} shows how the width
minimizing value of $n$ increases with $N$ in the BVHK
case with $f(x) = 1\{x<1/3\}$.
These are widths $W=2H^\ben(n,N/n)$ using
equation~\eqref{eq:bennettwidth}, not half widths.
Some are larger than $1$, because the Bennett
width does not reflect that we would intersect the intervals with $[0,1]$.
The increase in $n$ is quite slow. By $N=1024$
the best RQMC sample size is only $n=8$
and this remains the optimal sample size up to $N=4096$.
The value $N^{2/3}W$ very quickly becomes constant.
For non-integer $\theta$ (not shown here) the
scaled widths fluctuate % was fluctuation
within a finite range
instead of approaching a constant.

The result in Table~\ref{tab:bvhkcase} for $N=1024$
is to choose $n=8$ which violates the guidance that
the best $n$ is at most $7$ when $N=1024$ and $\alpha=0.05$.
The reason is that this
table restricts to powers of $2$, and $n=8$ is better
than $n=4$.

\begin{table}[t]\centering
\begin{tabular}{crrcc}
\toprule
   $\log_2(N)$&        $N$ &  $n$ & $W$ & $N^{2/3}W$\\
   \midrule
   \phz0&         1&   1& 5.02e$+$00 &5.020114\\
   \phz4&        16&   2& 7.60e$-$01 &4.826379\\
   \phz7&       128&   4& 1.90e$-$01 &4.826379\\
  10&      1,024&   8& 4.75e$-$02 &4.826379\\
  13&      8,192&  16& 1.19e$-$02 &4.826379\\
  16&     65,536&  32& 2.97e$-$03 &4.826379\\
  19&    524,288&  64& 7.42e$-$04 &4.826379\\
  22&   4,194,304& 128& 1.86e$-$04 &4.826379\\
  25&  33,554,432& 256& 4.64e$-$05 &4.826379\\
  28& 268,435,456& 512& 1.16e$-$05 &4.826379\\
 \bottomrule
\end{tabular}
\caption{\label{tab:bvhkcase}
As $N$ varies for $f(x)=1\{x<1/3\}$, this table shows the
value of $n$ that minimizes the width of Bennett's
interval, along with the minimizing width $W$
and $N^{2/3}W$. Each row corresponds to a value of $N$
for which $n$ has increased.}
\end{table}

Our second example is $f(x)=xe^{x-1}$. This function
is bounded between $0$ and $1$.  We chose it because it
is smooth enough to have $\theta=3$ while also not
a polynomial or an antisymmetric function which would
make the problem artificially easy for some methods. 
Because of the stratification, we know that
\begin{align}\label{eq:varsobol1d}
\var(\hat\mu) =\frac1{n^2}\sum_{\ell=1}^n\var( f(x_\ell))
\end{align}
where $x_\ell\sim\dunif[(\ell-1)/n,\ell/n]$ are independent.
Also
\begin{align}\label{eq:varonesmooth}
\var(f(x_\ell))=n\int_{(\ell-1)/n}^{\ell/n}f(x)^2\rd x
-\Bigl(n\int_{(\ell-1)/n}^{\ell/n}f(x)\rd x\Bigr)^2
\end{align}
which can be computed for finite $n$ using the closed form antiderivatives of $f$ and $f^2$.  
%Using~\eqref{eq:varonesmooth}
%in~\eqref{eq:varsobol1d} we get a value that increases
%rapidly towards the asymptotic expression but differs
%meaningfully for $n\le 8$.
We know that as $n$ increases through powers of 2
$$
\var(\hat\mu) \asymp
\frac1{12n^3}\int_0^1f'(x)^2\rd x = \frac{5-e^{-2}}{48n^3}
$$
but the exact variances are somewhat different for $n\le8$.

Table~\ref{tab:smoothcase} shows how the optimal sample sizes 
$n$ grows with $N$ for our smooth example 
$f(x)=xe^{x-1}$.
An RQMC sample size of $n=8$ becomes
optimal for $N=4{,}096$ and remains optimal up
to $N=32{,}768$.  The values $N^{3/4}W$
settle down rapidly as $N$ increases.
Consistent with the results of
Section~\ref{sec:asymptotic}, the values $n$
grow at a slower rate with $N$ than in the
BVHK case while the widths $W$ decrease at
a faster rate.

\begin{table}\centering
\begin{tabular}{crrcc}
\toprule
 $\log_2(N)$ &  $N$ &  $n$ & $W$ & $N^{3/4}W$ \\
\midrule
  0&         1&   1& 4.00e$+$00& 4.003728\\
  4&        16&   2& 5.17e$-$01& 4.138086\\
  8&       256&   4& 6.52e$-$02& 4.175717\\
 12&      4,096&   8& 8.17e$-$03& 4.185407\\
 16&     65,536&  16& 1.02e$-$03& 4.187848\\
 20&   1,048,576&  32& 1.28e$-$04& 4.188459\\
 24&  16,777,216&  64& 1.60e$-$05& 4.188612\\
 28& 268,435,456& 128& 2.00e$-$06& 4.188651\\
\bottomrule
%      m         N   n            w  N^3/4w 
%[1,]  0         1   1 4.003728e+00 4.003728
%[2,]  4        16   2 5.172607e-01 4.138086
%[3,]  8       256   4 6.524558e-02 4.175717
%[4,] 12      4096   8 8.174624e-03 4.185407
%[5,] 16     65536  16 1.022424e-03 4.187848
%[6,] 20   1048576  32 1.278216e-04 4.188459
%[7,] 24  16777216  64 1.597829e-05 4.188612
%[8,] 28 268435456 128 1.997304e-06 4.188651
% old numbers for f(x) = x.  The same best n
% \toprule
%       $\log_2(N)$&        $N$ &  $n$ & $W$ & $N^{3/4}W$\\
% \midrule
%   \phz0&         1 &  1 &4.03e$+$00 &4.027454\\
%   \phz4&        16 &  2 &5.03e$-$01 &4.027454\\
%   \phz8&       256 &  4 &6.29e$-$02 &4.027454\\
%  12&      4,096 &  8 &7.87e$-$03 &4.027454\\
%  16&     65,536 & 16 &9.83e$-$04 &4.027454\\
%  20&   1,048,576 & 32 &1.23e$-$04 &4.027454\\
%  24&  16,777,216 & 64 &1.54e$-$05 &4.027454\\
%  28& 268,435,456 &128 &1.92e$-$06 &4.027454\\
%  \bottomrule
\end{tabular}
\caption{\label{tab:smoothcase}
As $N$ varies for $f(x)=xe^{x-1}$, this table shows the
value of $n$ that minimizes the width of Bennett's
interval, along with the minimizing width $W$
and $N^{3/4}W$. Each row corresponds to a value of $N$
for which $n$ has increased.}
\end{table}

\subsection{Ridge functions}

While the betting intervals are asymptotically of the
same width as we would get from Bennett's formula, here
we investigate what choices of $n$ and $R$ give the
narrowest interval widths for some finite $N$. 
These computations produce actual intervals
that we then intersect with $[0,1]$.
It is well known that smoothness and dimension affect
the performance of RQMC methods.
We choose to use ridge functions because that makes it
straightforward to vary both the dimension and
smoothness of our test functions.

Let $\Phi$ be the $\dnorm(0,1)$ cumulative distribution function.
Then for $\bsx\sim\dunif[0,1]^d$ the vector
$\bsz=\Phi^{-1}(\bsx)$ componentwise,
has the $\dnorm(0,I)$ distribution over $\real^d$.
Then $\bsone^\tran\bsz/\sqrt{d}\sim\dnorm(0,1)$.
Our ridge function
$$
f(\bsx) =g\Bigl( \frac{\bsone^\tran\Phi^{-1}(\bsx)}{\sqrt{d}}\Bigr).
$$
then has the same mean and variance and differentiability properties
in any dimension $d$.  We choose
\begin{align*}
g_\jmp(v) &= 1\{v\ge1\},\\
g_\knk(v) & = \frac{\min(\max(-2,v),1)+2}3,\\
g_\smo(v) & = \Phi(v+1),\quad\text{and}\\
g_\fin(v) & = \min\Bigl(1,\frac{\sqrt{\max(v+2,0)}}2\Bigr).
\end{align*}

All of the above choices give ridge functions bounded between $0$ and $1$.
They span a very wide range of QMC regularity conditions.
Using $g_\jmp$ makes $f$ discontinuous
and it has infinite variation in
the sense of Hardy and Krause for $d\ge2$. 
Using $g_\knk$ makes $f$ have infinite variation in
the sense of Hardy and Krause for $d\ge3$ and yet
it has been seen to be a more `QMC friendly' integrand
than $g_\jmp$ in high dimensions because it has an
asymptotically bounded mean dimension \cite{mdridge},
meaning that it is dominated by low dimensional 
interactions.  It is continuous, has one weak derivative
but is not differentiable everywhere. Using $g_\smo$ makes $f$ infinitely
differentiable which is generally a very QMC friendly
property.  Using $g_\fin$ gives $f$ a property
similar to some integrands in financial option valuation, namely an
unbounded derivative where $v=-2$.

For each of these four ridge functions, we 
considered budgets $N=2^K$ for $K\in\{8,10,12,14,16\}$
with RQMC sample sizes $n=2^k$ for $k\in\{0,1,2,3,4,5,6\}$
in dimensions $d\in\{1,2,4,16\}$.  All of those
conditions were repeated $20$ times independently. In each
repetition we computed the width of the HBCI, the width of the EBCI, and the width of the CLT window.

The HBCI were narrower than the corresponding
EBCI about 92.1\% of the time.
The CLT intervals were narrower than the corresponding
HBCI 99.6\% of the time. The few times
that the CLT intervals were wider were predominantly
from cases with $n=1$. The EBCI were never narrower than the CLT intervals.
In some cases, the EBCI
had width $1$. This only happened for $N=256$ and
$n=32$ or $64$. Those cases have small replicate
numbers $R=8$ and $R=4$, respectively.

The HBCI are most interesting because they
have non-asymptotic coverage and are usually narrower
than the EBCI.
We made an extensive graphical exploration of their width
versus $n$ (for every combination of dimension, ridge function and $N$).
We similarly explored width versus budget $N$ and width
versus dimension $d$, in each case at all levels of all of the other
variables.  

It was evident that the dimension $d$ made
hardly any difference to the width of the confidence
intervals.  An additive model
using factors equal to the dimension $d$, the RQMC
sample size $n$, the function $f$ and the budget $N$
explained 94.09\% of the variance in the logarithm
of the HBCI width.  That same model without
dimension explained 94.05\%  of this variance.
It was surprising to see no dimension effect
for the HBCI width in our ridge functions.
We knew that the mean and variance of $f(\bsx)$ for
any of those ridge functions does not depend on $d$
but that does not imply that averages over $n$ RQMC
points would be invariant to $d$.
Because $d$ made no material difference to the
widths, we pooled information from all dimensions $d$.

It was clear graphically, and not surprising, that increasing $N$ by a factor of $4$ always narrows the HBCI width by a noticeable amount.
The interesting non-monotone patterns were from the way that the
width depends on $n$ with the other variables fixed. 
Figure~\ref{fig:meanwidths} shows the mean widths of
HBCI versus $n$
for all four ridge functions and all five budgets $N$.

\begin{figure}[t!]
    \centering
    \includegraphics[width=0.9\linewidth]{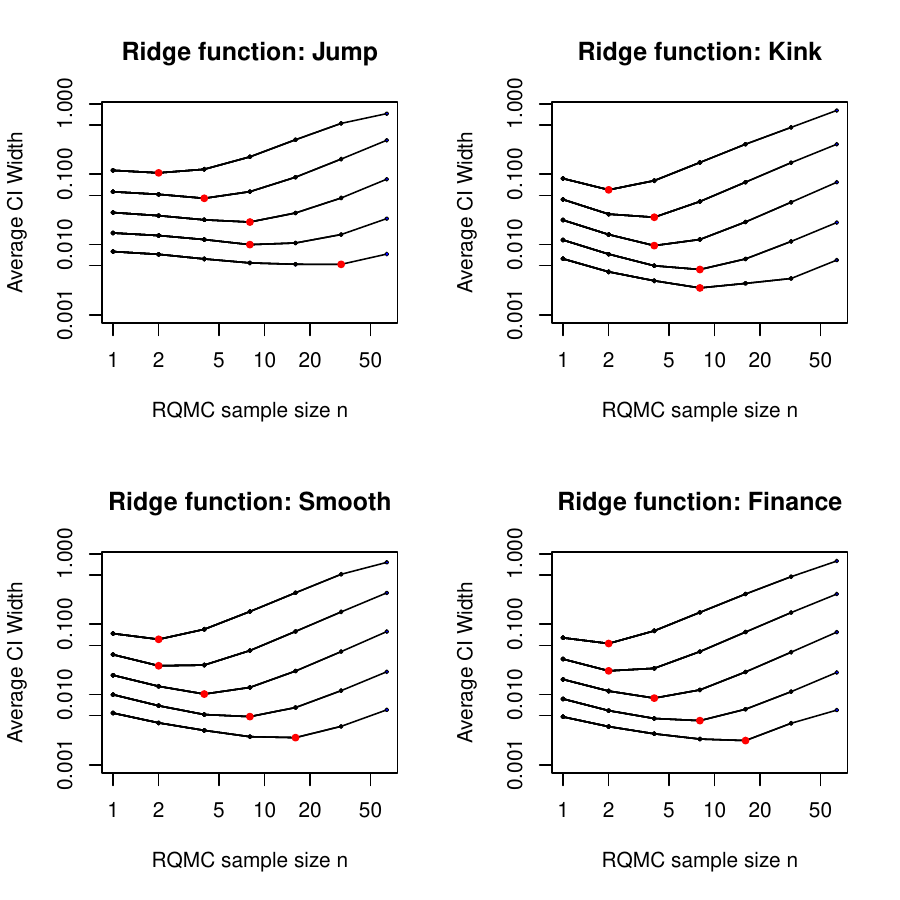}
    \caption{
Mean confidence interval width for the HBCI
as a function of the RQMC sample size $n$.
There is one panel for each of the $4$ ridge functions.
Within a panel the curves are for $N=2^{2r}$
for $r=4,5,6,7,8$ from top to bottom. The smallest
mean width for each $N$ is marked with a larger circle.
The means are taken over all $20$ replicates and
all $4$ dimensions used.
    }
    \label{fig:meanwidths}
\end{figure}

We see in Figure~\ref{fig:meanwidths} some things that
align very well with our analysis and some that
were not anticipated.
The optimal samples sizes $n$ are very small and
they grow slowly with the budget $N$ as expected.
The least smooth ridge function $g_\jmp$ has
the fastest growing $n$ as expected.
The interval widths for $g_\fin$ and $g_\smo$
were very close to each other 
%and those two ridge integrands had identical optimal $n$
despite having very different smoothness
and then they had identical optimal $n$ at each $N$.
For $N=2^{10}$ the optimal $n$ for $g_\knk$
is larger than the one for $g_\smo$ as expected
but at $N=2^{16}$, it is $g_\smo$ that has
the larger optimal $n$.

Our expectations were based on asymptotic
considerations.
First, we used the fact that the HBCI
widths asymptotically approach the optimal
EBCI widths from Bennett's
inequality.  Second, we supposed that the RQMC variance
is asymptotic to $\sigma_0^2n^{-\theta}$ apart from logarithmic
factors. We also used the finite sample Bennett
widths from~\eqref{eq:benh} because
the asymptotic expression $\sigma_n\sqrt{2\log(2/\alpha)/R}$
does not have a meaningful optimum.
The empirical results strongly confirm our
asymptotic findings that small $n$ is best overall,
but among such small $n$, it is less clear
how to choose $n$ based on smoothness of $f$.

\subsection{Functions with known RQMC variance}

Here we return to the functions
from Section~\ref{sec:semiempirical} with
known non-asymptotic RQMC variance.
They allow us to directly
compare the widths of the HBCI to
the value we get from Bennett's formula. We know
that the HBCI have lengths that
asymptotically match Bennett's formula, but
we also want to study how those lengths
approach Bennett's formula in some examples.

Using the known variances, we can compute the
half width from Bennett's inequality~\eqref{eq:bennettwidth}
for our $\sigma^2_n$ and $R=N/n$, double it,
and then divide the average HBCI width of 20 independent trials
by this value. The results are shown  in Figure~\ref{fig:widthstoeb}. For our two example
functions we see that as $R=N/n$ increases (so $n$ decreases
for fixed $N$)
the ratio of widths decreases and approaches the
theoretically expected value of $1$.
At the smallest $R$ (largest $n$) the HBCI were two or three times as wide as the
Bennett intervals for the discontinuous function.
They were somewhat wider for the smooth function.

\begin{figure}[t!]
    \centering
    \includegraphics[width=0.9\linewidth]{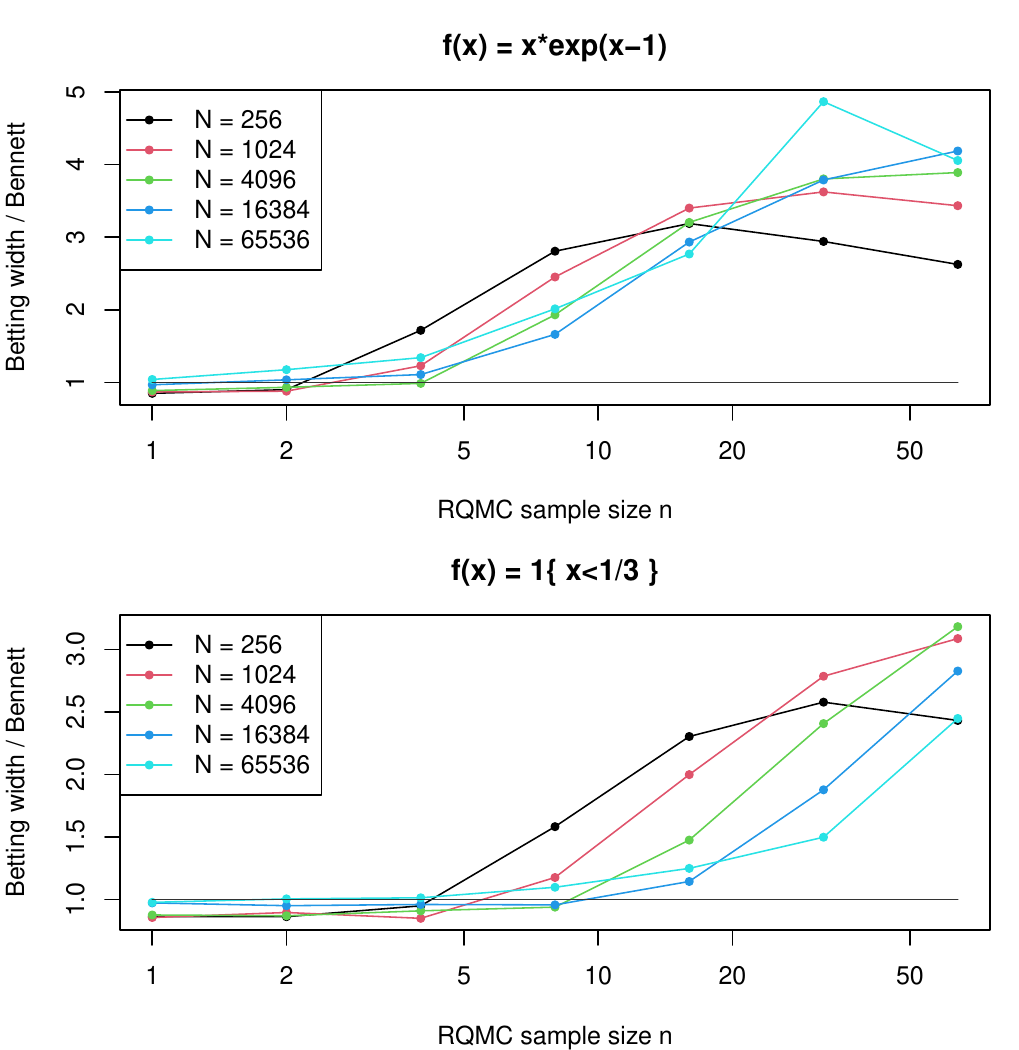}
    \caption{Ratio of mean HBCI width to  width from Bennett's inequality as a function of the RQMC sample size $n$ for the given budgets $N$. The means for the HBCI are taken over the 20 independent replicates. }
    \label{fig:widthstoeb}
\end{figure}

The HBCI are computed without knowledge
of the exact variance $\sigma^2_n$ that the Bennett
interval uses.  They are instead based on estimates
of that variance and they must therefore allow for
uncertainty in the variance estimate.  It is then not
surprising that the resulting intervals are wider
than the Bennett ones. We see in Table~\ref{tab:benvsbetn} that
the HBCI's minimum widths were
not found at larger $n$ than the minimizers
of the Bennett intervals. They were either equal to
or half as large as $n_*$. We also see that empirical
best values of $n$ are no larger for the smooth
function than for the discontinuous one, as expected.

\begin{table}
\centering
\begin{tabular}{rrrrr}
\toprule
&\multicolumn{2}{c}{$1\{x<1/3\}$}&\multicolumn{2}{c}{$xe^{x-1}$}\\
\midrule
$N$ & $n_*$ & $n_{\mathrm{opt}}$ & $n_*$ & $n_{\mathrm{opt}}$\\
\midrule
256 &  4   & 4 & 4 &2\\
1,024&  8   & 4 & 4 &4\\
4,096 & 8   & 8 & 8 &4\\
16,384 & 16 & 16 & 8 &8\\
65,536 & 32 & 16 &16 &8\\
\bottomrule
\end{tabular}
\caption{\label{tab:benvsbetn}
For the given budgets $N$, we show $n_*$ as the
minimizer of the Bennett width \eqref{eq:benh} 
over $n=2^k$ for integer $k$
%for the function $f(x) = 1\{x<1/3\}$. Then 
and $n_{\opt}$ is
the value of $n$ that minimized the average
width of the HBCI
for the given $N$. These are shown for functions
$f(x) =1\{x<1/3\}$ and $f(x)=xe^{x-1}$.
}
\end{table}

\section{Discussion}\label{sec:discussion}
In favorable cases
the RQMC variance is $\tilde O(n^{-3})$ and then
using a fixed number $R$ of replicates of
point sets with $n\to\infty$ we can estimate
the integral $\mu$ with a standard deviation 
of $\tilde O(N^{-3/2})$ for $N=nR$ along with
an unbiased estimate of the variance of that estimate.
It remains very difficult to get a confidence 
interval of width $\tilde O(N^{-3/2})$ \cite{err4qmc}.
Using Student's $t$ with a small $R$ gave good
results in an extensive simulation by \cite{LEcEtal24a}
but that success is not yet theoretically understood.
For an integrand subject to known bounds, we can get
a non-asymptotic confidence interval by using $R$
replicates of $n$ RQMC points using a
predictable plug-in empirical Berstein confidence
interval.  There, if $\sigma^2_n=\Theta(n^{-\theta})$
then the optimal $n$ is $\Theta(N^{1/(\theta+1)})$
and the resulting confidence intervals have width
$\Theta(N^{-\theta/(\theta+1)})$.
While it remains to find a way to choose the optimal $n$
empirically, we have strong theoretical guidance from 
equation~\eqref{eq:guidance}
that we can interpret as an upper bound on $n$ for the
oracle which is then a reasonable upper bound for HBCIs which have
to estimate the RQMC variance. When $n$ must be a
small power of $2$, we can rule out many suboptimal choices.

We have noted several times that there is an asymptotic
equivalence between the Bennett interval widths and those
of the EBCI and HBCI intervals.  In our setting that equivalence is as $R\to\infty$
for fixed $n$. In the RQMC context, the best $n$ is quite small.
Then for large $N$ it is not surprising that an asymptote as $R\to\infty$
is very predictive of our results.

It also remains to find a good way to
use RQMC in confidence sequences as opposed
to confidence intervals.   We might
set up $R$ independent infinite sequences of RQMC
points and then stop them at the first $n$
where they provide a $1-\alpha$ interval
narrower than some $\epsilon$.  In a task
like that it makes sense to only use
$n=2^k$ for integers $k$. 
That is partly because powers of two are good
for scrambled Sobol' points, but also because
in RQMC it is generally advisable to study
sample sizes $n$ that grow geometrically
not arithmetically \cite{sobo:1998}.  The
idea is that if $n$ is not large enough to get
a good answer, then $n+1$ is unlikely to be
much of an improvement.
A confidence sequence based on geometrically
growing $n$ would not have to be as wide as one for arithmetically
growing $n$.

For scrambled Sobol' points it is strongly
advisable to take $n=2^k$ because those are
the best sample sizes and they can even give
a better convergence rate. The Halton sequence \cite{Hal60} does
not have strongly superior sample sizes and so $n$ need
not be a power of two for it.  The 
two scrambles we considered for Sobol' points can
also be applied to the Halton sequence getting
$\var(\hat\mu)=o(1/n)$ without requiring special
sample sizes $n$ \cite{haltongain}.
There are however no published non-trivial settings
where scrambled Halton points can obtain $\theta>2$.

\section*{Acknowledgments}

We thank Aaditya Ramdas and Ian Waudby-Smith for helpful
discussions. We thank an anonymous reviewer for helpful comments.
ABO was supported by the U.S.\ National Science Foundation
grant DMS-2152780.
FJH and AGS were supported by U.S.\ National Science Foundation
grant DMS-2316011.
AGS was supported by the U.S. DOE Office of Science Graduate Student Research Program.
\bibliographystyle{plain}
\bibliography{FJH25copy,FJHown25copy,ebci4rqmc}
\end{document}